\documentclass[12pt, reqno]{amsart}
\usepackage{amsmath, amsthm, amscd, amsfonts, amssymb, graphicx, color}
\usepackage[bookmarksnumbered, colorlinks, plainpages]{hyperref}
\hypersetup{colorlinks=true,linkcolor=red,anchorcolor=green,citecolor=cyan,urlcolor=red, filecolor=magenta,pdftoolbar=true}

\newcommand{\arXiv}[1]{\href{http://arxiv.org/abs/#1}{\texttt{arXiv:#1}}}
\newcommand{\Hal}[1]{\href{https://hal.archives-ouvertes.fr/#1}{\texttt{HAL:#1}}}

\newcommand{\cH}{{\mathcal H}}

\def\rd{\mathrm{d}}
\def\ri{\mathrm{i}}
\def\re{\mathrm{e}}

\newcommand{\R}{{\mathbb R}}
\newcommand{\C}{{\mathbb C}}

\newtheorem{theorem}{Theorem}[section]
\newtheorem{lemma}[theorem]{Lemma}

\theoremstyle{definition}

\theoremstyle{remark}
\newtheorem{remark}[theorem]{Remark}
\numberwithin{equation}{section}

\newcommand{\bi}{\begin{itemize}}
\newcommand{\ei}{\end{itemize}}
\newcommand{\bd}{\begin{description}}
\newcommand{\ed}{\end{description}}
\newcommand{\be}{\begin{enumerate}}
\newcommand{\ee}{\end{enumerate}}

\newcommand{\dis}{\displaystyle}

\def\bc{\begin{center}}
\def\ec{\end{center}}

\newcommand{\tand}{\text{and}}

\newcommand{\lsum}{\sum\limits}
\newcommand{\lint}{\int\limits}

\def\z{\zeta}
\def\eps{\epsilon}

\def\l{\left}
\def\r{\right}
\def\bl{\bigl}
\def\br{\bigr}

\begin{document}
\setcounter{page}{40}

%\hspace{-0.45cm}{\includegraphics*[keepaspectratio=true,scale=0.4]{AOT.jpg}}\\
%\noindent
%{\footnotesize \qquad Adv. Oper. Theory 3 (2018), no. 2, 40--48\\
%\textcolor[rgb]{0.00,0.00,0.84}{http://doi.org/10.22034/aot.1705-1167}\\
%ISSN: 2538-225X (electronic)\\
%\textcolor[rgb]{0.00,0.00,0.84}{http://aot-math.org}\\[.3in]}

\title[Integral representations and asymptotic behaviour]{Integral representations and asymptotic behaviour of a Mittag-Leffler type function of two variables}

\author[C. Lavault]{Christian Lavault}
\address{LIPN, CNRS UMR 7030, Universit\'e Paris 13, Sorbonne Paris Cit\'e, F-93430 Villetaneuse, France.}
\email{\textcolor[rgb]{0.00,0.00,0.84}{lavault@lipn.univ-paris13.fr}}

\dedicatory{{\rm Communicated by C. Lizama}}

\let\thefootnote\relax\footnote{Copyright 2016 by the Tusi Mathematical Research Group.}

\subjclass[2010]{Primary 32A25; Secondary 45P05, 32D99.}

\keywords{Generalized two-parametric Mittag-Leffler type functions of two variables; Integral representations; Special functions; Hankel's integral contour; Asymptotic expansion formulas.}

\date{Received: May 26, 2017; Accepted: Oct. 18, 2017.}

\begin{abstract}
Integral representations play a prominent role in the analysis of entire functions. The representations of generalized Mittag-Leffler type functions and their asymptotics have been (and still are) investigated by plenty of authors in various conditions and cases.

The present paper explores the integral representations of a special function extending to two variables the two-parametric Mittag-Leffler type function. Integral representations of this functions within different variation ranges of its arguments for certain values of the parameters are thus obtained. Asymptotic expansion formulas and asymptotic properties of this function are also established for large values of the variables. This yields  corresponding theorems providing integral representations as well as expansion formulas.
\end{abstract}
\maketitle

\section{Definition and notation} \label{def}
Let the power series
\[
E_{\alpha,\beta}(z) := \lsum_{n=0}^\infty \frac{z^n}{\Gamma(\alpha n + \beta)}\ \qquad (\alpha, \beta\in \C;\ \Re(\alpha) > 0)\]
define the two-parametric Mittag-Leffler function (or M-L function for short)~\cite{Wiman05}. For the first parameter $\alpha$ with positive real part and any non restricted complex value of the second parameter $\beta$, the function $E_{\alpha,\beta}(z)$ is an entire functions of $z\in \C$ of order $\rho = 1/\Re(\alpha)$ and type $\sigma = 1$ (see, e.g, \cite[Chap.~4]{GoKiMaRo14}, \cite{HumAgr53}, \cite[\S1.1]{Lavault17}).

From here on, since we are concerned with integral representations and asymptotic expansions of generalized two-parametric M-L type functions, we shall restrict our attention to positive real-valued parameters $\alpha$ and $\beta$. Besides, the function $E_{\alpha,\beta}(z)$ of one variable $z\in \C$ will also be denoted for simplicity by $E_\alpha(z; \beta)$ in the proof of Lemma~\ref{lem1}, according to the notation used by D\v{z}rba\v{s}jan in~\cite{Djrba60, Djrba66}.

Thus, the two-parametric M-L function of one variable $z\in \C$\ extends to the generalized M-L type function $E_{\alpha,\beta}(x,y; \mu)$ of {\em two variables} $x, y\in \C$~\cite{Djrba60, Djrba66}. Provided that $\alpha, \beta > 0$, it is also an entire function defined by the double power series~\cite{GaMaKa13, OgorYasha10}
\begin{equation} \label{def_ml}
E_{\alpha,\beta}(x,y; \mu) := \lsum_{n,m=0}^\infty \frac{x^n y^m}{\Gamma(n\alpha + m\beta + \mu)}\ %
\qquad (\alpha, \beta\in \R,\, \alpha, \beta > 0,\ \mu\in \C),
\end{equation}
in which the arbitrary parameter $\mu$ takes in general complex values.

Following~\cite{Djrba60, Djrba66} (see also, e.g., \cite{GoKiMaRo14, GorLouLuc02, HaMaSa11, Temme96}, \cite[\S1.2~\&~ App.~A, C~\&~D]{Lavault17}, and references therein), $E_{\alpha,\beta}(x,y; \mu)$ can be written in terms of {\em Hankel's integral representations} depending on the variation ranges of the arguments, thereby as special cases of the Fox $H$-function. The {\em Hankel path} considered further in Lemma~\ref{lem1} to~\ref{lem3} and in Theorem~\ref{main} is denoted by $\gamma(\eps;\eta) := \bl\{0 < \eta\le \pi,\ \eps > 0\br\}$, defining a contour integral oriented by non-decreasing $\arg \z$. It consists of the following two parts depicted for instance in~\cite[Fig.~1 to~4]{GorLouLuc02}):
\be
\item the two rays $S_{\eta} =  \bl\{\arg \z = \eta,\ |\z|\ge \eps\br\}$ and $S_{-\eta} = %
\bl\{\arg \z = -\eta,\ |\z|\ge \eps\br\}$;
\item the circular arc $C_\eta(0;\eps) =  \bl\{|\z| = \eps,\ -\eta\le \arg \z\le \eta\br\}$.
\ee
If $0 < \eta < \pi$, then the contour $\gamma(\eps;\eta)$ divides the complex $\z$-plane into two unbounded regions, namely $\Omega^{(-)}(\eps;\eta)$ to the left of $\gamma(\eps ;\eta)$ by orientation and $\Omega^{(+)}(\eps;\eta)$ to the right of the contour. If $\eta = \pi$, then the contour consists of the circle $\br\{|\z| = \eps\bl\}$ and of the ray $-\infty < \z\le -\eps$ ($\eps > 0$), which is a two-way path (one in each direction) along the real line. More precisely, this {\em  keyhole} or Hankel contour is a path from $-\infty$ inbound along the real line to $-\eps < 0$, counterclockwise around a circle of radius $\epsilon$ at 0, back to $-\eps$ on the real line, and outbound back to $-\infty$ along the real line.

\section{Integral representations} \label{rep}
This section provides a few lemmas, which show various integral representations of the generalized M-L type function~\eqref{def_ml} corresponding to different variation ranges of the two arguments.

\begin{lemma} \label{lem1}
Let $0 < \alpha, \beta < 2$\ and $\alpha\beta < 2$. Let $\mu$ be any complex number and let $\eta$ satisfy the condition
\begin{equation} \label{cond1}
\pi \alpha\beta/2 < \eta\le \min\bl(\pi, \pi\alpha \beta\br).
\end{equation}

If $x\in \Omega^{(-)}(\eps_\alpha;\eta_\alpha)$\ and $y\in \Omega^{(-)}(\eps_\beta;\eta_\beta)$, where $\eps_\alpha := \eps^{1/\beta}$, $\eps_\beta := \eps^{1/\alpha}$ and $\eta_\alpha := \eta/\beta$, $\eta_\beta := \eta/\alpha$, then the Hankel integral representation holds
\begin{equation} \label{mlrep1}
E_{\alpha,\beta}(x,y;\mu) = \frac{1}{2\pi\ri} \frac{1}{\alpha\beta} \lint_{\gamma(\eps;\eta)} %
\frac{\re^{\z^{1/(\alpha \beta)}} \z^{\frac{\alpha + \beta + 1 - \mu}{\alpha \beta} - 1}} %
{(\z^{1/\alpha} - y)(\z^{1/\beta} - x)}\, \rd \z.
\end{equation}
\end{lemma}
\begin{proof}
First, let $|x| < \eps_\alpha$. Taking into account the fact that $\eps_\alpha = \eps^{1/\beta} = %
\l(\eps_\beta^\alpha\r)^{1/\beta} = \eps_\beta^{\alpha/\beta}$ yields next the inequality~\eqref{ineq1}
\begin{equation} \label{ineq1}
\sup_{\z\in \gamma(\eps_\beta;\eta_\beta)} \bl|x\z^{-\alpha/\beta}\br| < 1.
\end{equation}
From definition~\eqref{def_ml}, the expansion of $E_{\alpha,\beta}(x,y; \mu)$ may be rewritten as follows in terms of the corresponding two-parametric M-L function $E_\beta(y; n \alpha + \mu)$ of one variable,
\begin{flalign} \label{oneparml}
E_{\alpha,\beta}(x,y ; \mu) &= \lsum_{n=0}^\infty \lsum_{m=0}^\infty \frac{x^n y^m} %
{\Gamma(n \alpha + m \beta + \mu)} \nonumber\\
&= \lsum_{n=0}^\infty x^n \lsum_{m=0}^\infty \frac{y^m}{\Gamma(m \beta + (n \alpha + \mu))} %
= \lsum_{n=0}^\infty x^n E_\beta(y; n \alpha + \mu).
\end{flalign}
Under the assumptions of Lemma~\ref{lem1}, it is possible to use the known integral representation of $E_\beta(y; \alpha n + \mu)$\ (see, e.g., \cite[Eq.~(2.2)]{Djrba66}) by taking the above $\eps_\beta$\ and $\eta_\beta$\ as the parameters defining the Hankel contour, which is admissible according to inequalities~\eqref{cond1}. For $y\in \Omega^{(-)}(\eps_\beta;\eta_\beta)$, and provided that $\eta_\beta = \eta/\alpha$, the following representations holds from the integral representation of $E_\beta(y; \alpha n + \mu)$
\begin{flalign} \label{intrep0}
E_{\alpha,\beta}(x,y; \mu) &= \lsum_{n=0}^\infty x^n E_\beta(y; n \alpha + \mu) \nonumber\\
&= \lsum_{n=0}^\infty x^n\, \frac{1}{2\pi\ri} \frac{1}{\beta} \lint_{\gamma(\eps_\beta;\eta_\beta)} %
\frac{\re^{\z^{1/\beta}} \z^{\frac{1 - n \alpha - \mu}{\beta}}} {\z - y}\, \rd \z.
\end{flalign}
And by simplifying and using inequality~\eqref{ineq1} one gets
\begin{flalign}
E_{\alpha,\beta}(x,y; \mu) &= \frac{1}{2\pi\ri} \frac{1}{\beta} \lint_{\gamma(\eps_\beta;\eta_\beta)} %
\frac{\re^{\z^{1/\beta}} \z^{\frac{1-\mu}{\beta}}} {\z - y}\, \l(\lsum_{n=0}^\infty %
\l(x\z^{-\alpha/\beta}\r)^{n}\r)\, \rd \z \nonumber\\
&= \frac{1}{2\pi\ri} \frac{1}{\beta} \lint_{\gamma(\eps_\beta;\eta_\beta)} %
\frac{\re^{\z^{1/\beta}} \z^{\frac{1 + \alpha - \mu}{\beta}}} %
{(\z - y)(\z^{\alpha/\beta} - x)}\, \rd \z. \label{intrep1}
\end{flalign}

Now, by rewriting the above integral representation~\eqref{intrep1} along the suitable integral contour
$\gamma(\eps;\eta)$, we obtain
\begin{flalign*}
E_{\alpha,\beta}(x,y; \mu) &= \frac{1}{2\pi\ri} \frac{1}{\beta} \lint_{\gamma(\eps;\eta)} %
\frac{ \re^{\l(\xi^{1/\alpha}\r)^{1/\beta}} \l(\xi^{1/\alpha}\r)^{\frac{1 + \alpha - \mu}{\beta} } } %
{(\xi^{1/\alpha} - y)(\xi^{1/\beta} - x)}\, \frac{1}{\alpha}\, \xi^{\frac{1 - \alpha}{\alpha}}\, \rd \xi
\end{flalign*}
and get the desired integral representation~\eqref{mlrep1} set out in Lemma~\ref{lem1}
\begin{flalign*}
E_{\alpha,\beta}(x,y; \mu) &= \frac{1}{2\pi\ri} \frac{1}{\alpha\beta} \lint_{\gamma(\eps;\eta)} %
\frac{\re^{\xi^{1/(\alpha \beta)}} \xi^{\frac{\alpha + \beta + 1 - \mu}{\alpha \beta} - 1}} %
{(\xi^{1/\alpha} - y)(\xi^{1/\beta} - x)}\, \rd \xi.
\end{flalign*}
The above resulting integral is absolutely convergent and it is an analytic function of $x\in \Omega^{(-)}(\eps_\alpha;\eta_\alpha)$ and $y\in \Omega^{(-)}(\eps_\beta;\eta_\beta)$.

The open disk $D = \{|x| < \eps_\alpha\}$ is contained into the complex region $\Omega^{(-)}(\eps_\alpha;\eta_\alpha)$ for all values of $\eta_\alpha$ taken in the interval $\bl]\pi\alpha/2, \min(\pi,\pi \alpha)\br]$. Therefore, from the principle of analytic continuation Eq.~\eqref{mlrep1} is valid everywhere within the complex region $\Omega^{(-)}(\eps_\alpha;\eta_\alpha)$ and the lemma is established.
\end{proof}

\begin{lemma} \label{lem2}
Let $0 < \alpha, \beta < 2$\ and $\alpha\beta < 2$ Let $\mu$ be any complex number and let $\eta$ verify inequalities~\eqref{cond1},
\[
\pi \alpha\beta/2 < \eta\le \min\bl(\pi, \pi \alpha\beta\br).\]
If $x\in \Omega^{(-)}(\eps_\alpha;\eta_\alpha)$\ and $y\in \Omega^{(+)}(\eps_\beta;\eta_\beta)$, where $\eps_\alpha := \eps^{1/\beta}$, $\eps_\beta := \eps^{1/\alpha}$ and $\eta_\alpha := \eta/\beta$, $\eta_\beta := \eta/\alpha$, then the integral representation holds
\begin{equation} \label{mlrep2}
E_{\alpha,\beta}(x,y;\mu) = \frac{1}{\beta}\, \frac{ \re^{y^{1/\beta}} %
y^{\frac{1 + \alpha - \mu}{\beta}} } {y^{\alpha/\beta} - x} + \frac{1}{2\pi\ri} \frac{1}{\alpha\beta} %
\lint_{\gamma(\eps;\eta)} \frac{ \re^{\z^{1/(\alpha \beta)} } \z^{\frac{\alpha + \beta + 1 - \mu} %
{\alpha \beta} - 1} } {(\z^{1/\alpha} - y)(\z^{1/\beta} - x)}\, \rd \z.
\end{equation}
\end{lemma}
\begin{proof}
By assumption, the point $y$ is located to the right of the Hankel contour $\gamma(\eps_\beta;\eta_\beta)$, that is $y\in \Omega^{(+)}(\eps_\beta;\eta_\beta)$. Then, for any $\eps_{\beta_1} > |y|$, we have that $y\in \Omega^{(-)}(\eps_{\beta_1};\eta_\beta)$\ and $x\in \Omega^{(-)}(\eps_{\alpha_1};\eta_\alpha)$ for $\eps_{\alpha_1} = \eps^{1/\beta_1}$. Therefore, by~\eqref{intrep1} we get the integral representation
\begin{equation} \label{intrep2}
E_{\alpha,\beta}(x,y;\mu) = \frac{1}{2\pi\ri} \frac{1}{\beta} \lint_{\gamma(\eps_{\beta_1} ; %
\eta_\beta)} \frac{ \re^{\z^{1/\beta} } \z^{\frac{\alpha + \beta - \mu}{\beta}} } %
{(\z - y)(\z^{\alpha/\beta} - x)}\, \rd \z.
\end{equation}
On the other hand, if $\eps_\beta < |y| < \eps_{\beta_1}$, then $|\arg y| < \eta_\beta$ and, by Cauchy theorem,
\begin{flalign}  \label{intrep3}
E_{\alpha,\beta}(x,y;\mu) &= \frac{1}{2\pi\ri} \frac{1}{\beta}\!\! \lint_{\gamma(\eps_{\beta_1}; \eta_\beta) - \gamma(\eps_\beta;\eta_\beta)} \frac{ \re^{\z^{1/\beta}} %
\z^{\frac{1 + \alpha - \mu}{\beta}} } {(\z^{1/\alpha} - y)(\z^{1/\beta} - x)}\, \rd \z\nonumber\\
&= \frac{1}{\beta}\, \frac{ \re^{y^{1/\beta}} y^{\frac{1 + \alpha - \mu}{\beta}} }{y^{\alpha/\beta} - x}.
\end{flalign}
Hence, from Eqs.~\eqref{intrep2} and~\eqref{intrep3}, we obtain the integral representation~\eqref{mlrep2} and Lemma~\ref{lem2} follows.
\end{proof}

\begin{remark} \label{rem1}
Symmetrically, for $x\in \Omega^{(+)}(\eps_\alpha;\eta_\alpha)$, $y\in \Omega^{(-)}(\eps_\beta;\eta_\beta)$\ and under the assumptions of Lemma~\ref{lem2}, the integral representation of $E_{\alpha,\beta}(x,y ; \mu)$ is shown in a same manner to be
\begin{equation} \label{mlrep3}
E_{\alpha,\beta}(x,y;\mu) = \frac{1}{\alpha}\, \frac{ \re^{x^{1/\alpha}} %
x^{\frac{1 + \beta - \mu}{\alpha}} } {x^{\beta/\alpha} - y} + \frac{1}{2\pi\ri} \frac{1}{\alpha\beta} %
\lint_{\gamma(\eps;\eta)} \frac{ \re^{\z^{1/(\alpha \beta)} } \z^{\frac{\alpha + \beta + 1 - \mu} %
{\alpha \beta} - 1} } {(\z^{1/\alpha} - y)(\z^{1/\beta} - x)}\, \rd \z,
\end{equation}
by simply interchanging $\alpha$\ and $\beta$ in representation~\eqref{mlrep2}.
\end{remark}

\begin{lemma} \label{lem3}
Let $0 < \alpha, \beta < 2$\ and $\alpha\beta < 2$. Let $\mu$ be any complex number and let $\eta$ verify inequalities~\eqref{cond1}.
If $x\in \Omega^{(+)}(\eps_\alpha;\eta_\alpha)$\ and $y\in \Omega^{(+)}(\eps_\beta;\eta_\beta)$, where $\eps_\alpha := \eps^{1/\beta}$, $\eps_\beta := \eps^{1/\alpha}$\ and $\eta_\alpha := \eta/\beta$, $\eta_\beta := \eta/\alpha$, then the integral representation holds
\begin{flalign} \label{mlrep4}
E_{\alpha,\beta}(x,y;\mu) = \frac{1}{\alpha}\, \frac{ \re^{x^{1/\alpha}} %
x^{\frac{1 + \beta - \mu}{\alpha}} }{x^{\beta/\alpha} - y} &+ \frac{1}{\beta}\, \frac{\re^{y^{1/\beta}} %
y^{\frac{1 + \alpha - \mu}{\beta}} } {y^{\alpha/\beta} - x}\nonumber\\
& + \frac{1}{2\pi\ri} \frac{1}{\alpha\beta} \lint_{\gamma(\eps;\eta)} %
\frac{ \re^{\z^{1/(\alpha \beta)} } \z^{\frac{\alpha + \beta + 1 - \mu}{\alpha \beta} - 1} } %
{(\z^{1/\alpha} - y)(\z^{1/\beta} - x)}\, \rd \z.
\end{flalign}
\end{lemma}
\begin{proof}
By assumption, each of the points $x$ and $y$ lies on the right-hand side of the Hankel contours $\gamma(\eps_\alpha;\eta_\alpha)$ and $\gamma(\eps_\beta;\eta_\beta)$, respectively; that is in the two regions of the complex plane defined by $x\in \Omega^{(+)}(\eps_\alpha;\eta_\alpha)$ and $y\in \Omega^{(+)}(\eps_\beta;\eta_\beta)$\ (where the parameters $\eps_\alpha$ and $\eps_\beta$ correspond to $\eps$). Now, choose $\eps_1 > \eps$\ such that one of the coordinates is to the right of the contour and the other coordinate to its left (which is always possible provided that $x^\beta\neq y^\alpha$).

Let $x\in \Omega^{(-)}(\eps_{\alpha_1};\eta_\alpha)$\ and $y\in \Omega^{(+)}(\eps_{\beta_1};\eta_\beta)$ (i.e., $x < y$) for $\eps_{\alpha_1} = \eps^{1/\beta_1}$\ and $\eps_{\beta_1} = \eps^{1/\alpha_1}$. Then, by Eq.~\eqref{mlrep2} in Lemma~\ref{lem2}, we have the integral representation
\begin{equation} \label{intrep4}
E_{\alpha,\beta}(x,y ;  \mu) = \frac{1}{\beta}\, \frac{\re^{y^{1/\beta}} %
y^{\frac{1 + \alpha - \mu}{\beta}}} {y^{\alpha/\beta} - x} + \frac{1}{2\pi\ri} \frac{1}{\alpha \beta} %
\lint_{\gamma(\eps_1;\eta)} \frac{ \re^{\z^{1/(\alpha\beta)}} %
\z^{\frac{\alpha + \beta + 1 - \mu}{\alpha \beta} - 1} }{(\z^{1/\alpha} - y)(\z^{1/\beta} - x)}\, \rd \z.
\end{equation}

Upon changing the variable $\z^{1/\beta}$ for $t$ in~\eqref{intrep4}, the above integral may be rewritten under the form
\begin{equation} \label{changevar}
\frac{1}{2\pi\ri} \frac{1}{\alpha} \lint_{\gamma(\eps_1;\eta)} \frac{ \re^{t^{1/\alpha} } %
t^{\frac{1 + \beta - \mu}{\alpha}} } {(t - x)(t^{\beta/\alpha} - y)}\, \rd t.
\end{equation}
Now, when $\eps_\alpha < |x| < \eps_{\alpha_1}$, then $|\arg x| < \eta_\alpha$ and, by Cauchy theorem,
\begin{equation} \label{intrep5}
\kern-.4cm E_{\alpha,\beta}(x,y ; \mu) = \frac{1}{2\pi\ri} \frac{1}{\alpha} \!\!
\lint_{\gamma(\eps_{\alpha_1};\eta_\alpha)-\gamma(\eps_\alpha;\eta_\alpha)} %
\frac{ \re^{\z^{1/\alpha}} \z^{\frac{1 + \beta - \mu}{\alpha}} }{(\z^{\beta/\alpha} - y)(\z - x)}\, %
\rd \z = \frac{1}{\alpha}\, \frac{ \re^{x^{1/\alpha}} x^{\frac{1 + \beta - \mu}{\alpha}} } %
{x^{\beta/\alpha} - y}.
\end{equation}
Finally, from Eqs.~\eqref{intrep4} and~\eqref{intrep5} the representation~\eqref{mlrep4} holds true, and the lemma is established.
\end{proof}

\begin{lemma} \label{{lem4}}
If $\Re(\mu) > 0$, then the integral representations~\eqref{mlrep1}, \eqref{mlrep2}, \eqref{mlrep3} and \eqref{mlrep4} remain valid for $\alpha = 2$\ or $\beta = 2$.
\end{lemma}
\begin{proof}
The lemma follows immediately by passing to the limit with respect to the corresponding parameters in representations~\eqref{mlrep1}, \eqref{mlrep2}, \eqref{mlrep3} and~\eqref{mlrep4}.
\end{proof}

\section{Asymptotic behaviour} \label{as}
The asymptotic properties of the function $E_{\alpha,\beta}(x,y;\mu)$ for large values of $|x|$ and $|y|$ are of particular interest.

\begin{theorem} \label{main}
Let $0 < \alpha, \beta < 2$\ and $\alpha\beta < 2$. Let $\mu$ be any complex number and $\tau$ be any real number satisfying inequalities~\eqref{cond1}
\[
\pi \alpha\beta/2 < \tau\le \min\bl(\pi, \pi \alpha\beta\br).\]
Then, for all integer $r\ge 1$, the function $E_{\alpha,\beta}(x,y; \mu)$ verifies the following asymptotic formulas drawn from its integral representations whenever $|x|\to \infty$ and $|y|\to \infty$.
\bi
\item[1)]\ If $|\arg x|\le \tau/\beta$\ and $|\arg y|\le \tau/\beta$, then
\begin{flalign} \label{as1}
E_{\alpha,\beta}(x,y; \mu) &= \frac{1}{\alpha}\, \frac{ \re^{x^{1/\alpha}} %
x^{\frac{1 + \beta - \mu}{\alpha}} }{x^{\beta/\alpha} - y} + \frac{1}{\beta}\, \frac{\re^{y^{1/\beta}} %
y^{\frac{1 + \alpha - \mu}{\beta}} } {y^{\alpha/\beta} - x}\nonumber\\
+ \lsum_{n=1}^{r_\beta} &\lsum_{m=1}^{r_\alpha} \frac{x^{-n} y^{-m}}{\Gamma(\mu - n\alpha - m\beta)} %
+ o\l(|xy|^{-1} |x|^{-r_\alpha}\r) + o\l(|xy|^{-1} |y|^{-r_\beta}\r)
\end{flalign}

\item[2)]\ If $|\arg x|\le \tau/\beta$\ and $\tau/\alpha < |\arg y|\le \pi$, then
\begin{flalign} \label{as2}
E_{\alpha,\beta}(x,y; \mu) &= \frac{1}{\alpha}\, \frac{ \re^{x^{1/\alpha}} %
x^{\frac{1 + \beta - \mu}{\alpha}} }{x^{\beta/\alpha} - y}\nonumber\\
+ \lsum_{n=1}^{r_\beta} &\lsum_{m=1}^{r_\alpha} \frac{x^{-n} y^{-m}}{\Gamma(\mu - n\alpha - m\beta)} %
+ o\l(|xy|^{-1} |x|^{-r_\alpha}\r) + o\l(|xy|^{-1} |y|^{-r_\beta}\r)
\end{flalign}

\item[3)]\ If $\tau/\beta < |\arg x|\le \pi$\ and $|\arg y|\le \tau/\alpha$, then
\begin{flalign} \label{as3}
E_{\alpha,\beta}(x,y; \mu) &= \frac{1}{\beta}\, \frac{ \re^{y^{1/\beta}} %
y^{\frac{1 + \alpha - \mu}{\beta}} }{y^{\alpha/\beta} - x}\nonumber\\
+ \lsum_{n=1}^{r_\beta} &\lsum_{m=1}^{r_\alpha} \frac{x^{-n} y^{-m}}{\Gamma(\mu - n\alpha - m\beta)} %
+ o\l(|xy|^{-1} |x|^{-r_\alpha}\r) + o\l(|xy|^{-1} |y|^{-r_\beta}\r)
\end{flalign}

\item[4)]\ If $\tau/\beta < |\arg x|\le \pi$\ and $\tau/\alpha < |\arg y|\le \pi$, then
\begin{flalign} \label{as4}
E_{\alpha,\beta}(x,y; \mu) =& \lsum_{n=1}^{r_\beta} \lsum_{m=1}^{r_\alpha} \frac{x^{-n} y^{-m}} %
{\Gamma(\mu - n\alpha - m\beta)}\nonumber\\
&+ o\l(|xy|^{-1} |x|^{-r_\alpha}\r) + o\l(|xy|^{-1} |y|^{-r_\beta}\r).
\end{flalign}
\ei
\end{theorem}

\begin{proof}
The proof below focuses on the first case, since the proofs ot the three other cases are easily completed along the same lines as in case 1, that is as the proof of asymptotic formula~\eqref{as1}.

So, under the conditions required in case 1, i.e. $|\arg x|\le \tau/\beta$\ and $|\arg y|\le \tau/\alpha$, pick a real number $\theta$ satisfying the condition~\eqref{cond2}:
\begin{equation} \label{cond2}
\pi \alpha\beta/2 < \tau < \theta\le \min\bl(\pi, \pi\alpha\beta\br).
\end{equation}
It is easy to expand the equality
\begin{equation} \label{exp}
\frac{1}{(\z^{1/\beta} - x)(\z^{1/\alpha} - y)} = \lsum_{n=1}^{r_\beta} \lsum_{m=1}^{r_\alpha} %
\frac{ \z^{\frac{n - 1}{\beta} + \frac{m - 1}{\alpha}} }{x^n y^m} + %
\frac{ x^{r_\beta} \z^{ \frac{r_\alpha}{\alpha} } + y^{r_\alpha} \z^{ \frac{r_\beta}{\beta} } %
- \z^{ \frac{r_\alpha}{\alpha} + \frac{r_\beta}{\beta}} }{x^{r_\beta} y^{r_\alpha} %
(\z^{1/\beta} - x)(\z^{1/\alpha} - y) }\,.
\end{equation}

Set $\eps = 1$ in the representation~\eqref{mlrep4} in Lemma~\ref{lem3}. Then, to the right of the contour $\gamma(1;\theta)$ (i.e. within the complex region $\Omega^{(+)}(1;\theta)$), in view of expansion~\eqref{exp} and by Eq.~\eqref{mlrep4}, the integral representation of $E_{\alpha,\beta}(x,y; \mu)$ takes the form

\begin{flalign} \label{mlrep5}
E_{\alpha,\beta}(x,y; \mu) &= \frac{1}{\alpha}\, \frac{ \re^{x^{1/\alpha}} %
x^{\frac{1 + \beta - \mu}{\alpha}} }{x^{\beta/\alpha} - y} + \frac{1}{\beta}\, \frac{\re^{y^{1/\beta}} %
y^{\frac{1 + \alpha - \mu}{\beta}} } {y^{\alpha/\beta} - x}\nonumber\\
+& \lsum_{n=1}^{r_\beta} \lsum_{m=1}^{r_\alpha} \frac{1}{2\pi\ri} \frac{1}{\alpha\beta}\, %
\l(\lint_{\gamma(1;\theta)} \re^{\z^{1/(\alpha \beta)}} \z^{ \frac{\alpha + \beta + 1 - \mu}{\alpha \beta} - 1 + \frac{n - 1}{\beta} + \frac{m - 1}{\alpha} }\, \rd \z)\r)\, x^{-n} y^{-m}\nonumber\\
+& \frac{1}{2\pi\ri} \frac{1}{\alpha \beta} \lint_{\gamma(1;\theta)} \re^{ \z^{1/(\alpha \beta)} } %
\z^{ \frac{\alpha + \beta + 1 - \mu}{\alpha \beta} - 1 }\; \frac{ x^{r_\beta} \z^{\frac{r_\alpha}{\alpha}} %
+ y^{r_\alpha} \z^{\frac{r_\beta}{\beta}} - \z^{\frac{r_\alpha}{\alpha} + \frac{r_\beta}{\beta}} } %
{x^{r_\beta} y^{r_\alpha} (\z^{1/\beta} - x )(\z^{1/\alpha} - y) }\, \rd \z.
\end{flalign}
Now, the Hankel representation of the reciprocal gamma function is obtained through the suited Hankel contour $\cH$\ detailed in~\cite[Eq.~C3]{Lavault17}, \cite[Chap.~3, \S3.2.6]{Temme96}, etc., and written as the well-known contour integral formula
\[
\frac{1}{\Gamma(s)} = \lint_{\cH} \re^u u^{-s}\, \rd u\ \qquad (s\in \C,\ u > 0).\]
In the present setting, the integral contour is defined by $\cH_\theta = \gamma(1 ; \theta)$ ($\pi \alpha\beta/2 < \tau < \theta\le \min\bl(\pi, \pi \alpha \beta\br)$), according to the definition of the Hankel path $\gamma(\eps ; \eta)$ in Section~\ref{def} and the assumptions of the theorem.

As a consequence, the summand of the double sum, the second term in~\eqref{mlrep5}, satisfies the relation
\begin{flalign} \label{summ}
\frac{1}{2\pi\ri} & \frac{1}{\alpha\beta}\, \lint_{\gamma(1;\theta)} \re^{\z^{1/(\alpha \beta)}} %
\z^{ \frac{\alpha + \beta + 1 - \mu}{\alpha \beta} - 1 + \frac{n - 1}{\beta} + \frac{m - 1}{\alpha} }\, %
\rd \z = \frac{1}{2\pi\ri} \frac{1}{\alpha\beta}\, \lint_{\gamma(1;\theta)} \re^{\z^{1/(\alpha \beta)}} %
\z^{ \frac{1- \mu}{\alpha \beta} - 1 + \frac{n}{\beta} + \frac{m}{\alpha} }\, \rd \z\nonumber\\
&= \frac{1}{2\pi\ri} \frac{1}{\alpha\beta}\, \lint_{\gamma(1;\theta)} \re^{\z^{1/(\alpha \beta)}} %
\z^{ -\frac{1}{\alpha \beta}\bl(\mu - n \alpha - m \beta\br) + \frac{1}{\alpha \beta} - 1 }\, \rd \z %
= \frac{1}{\Gamma(\mu - n \alpha - m \beta)}\,.
\end{flalign}
Therefore, under the constraints resulting from inequalities~\eqref{cond2}, substituting the above summand~\eqref{summ} into representation~\eqref{mlrep5} yields the transformation
\begin{flalign} \label{mlrep6}
E_{\alpha,\beta}(x,y; \mu) &= \frac{1}{\alpha}\, \frac{ \re^{x^{1/\alpha}} %
x^{\frac{1 + \beta - \mu}{\alpha}} }{x^{\beta/\alpha} - y} + \frac{1}{\beta}\, \frac{\re^{y^{1/\beta}} %
y^{\frac{1 + \alpha - \mu}{\beta}} } {y^{\alpha/\beta} - x} + \lsum_{n=1}^{r_\beta} %
\lsum_{m=1}^{r_\alpha} \frac{x^{-n} y^{-m}}{\Gamma(\mu - \alpha n - \beta m)}\nonumber\\
+ \frac{1}{2\pi\ri} \frac{1}{\alpha \beta} & \lint_{\gamma(1;\theta)} %
\re^{ \z^{1/(\alpha \beta)} } \z^{ \frac{\alpha + \beta + 1 - \mu}{\alpha \beta} - 1 }\; %
\frac{ x^{r_\beta} \z^{r_\alpha/\alpha} + y^{r_\alpha} \z^{r_\beta/\beta} - %
\z^{ r_\alpha/\alpha + r_\beta/\beta } } %
{x^{r_\beta} y^{r_\alpha} (\z^{1/\beta} - x)(\z^{1/\alpha} - y) }\, \rd \z.
\end{flalign}
Next, expanding and simplifying the final term above makes the obvious sum of the three integrals $I_1 + I_2 + I_3$\ appear in Eq.~\eqref{mlrep6}; precisely,
\begin{flalign}
I_1 &= \frac{1}{2\pi\ri} \frac{1}{\alpha \beta}\, \lint_{\gamma(1;\theta)}\, %
\frac{ \re^{ \z^{1/(\alpha \beta)} } \z^{ \frac{\alpha + \beta + 1 - \mu}{\alpha \beta} - 1 %
+ r_\alpha/\alpha } } {y^{r_\alpha} (\z^{1/\beta} - x)(\z^{1/\alpha} - y)}\, \rd \z, \label{int1}\\
I_2 &= \frac{1}{2\pi\ri} \frac{1}{\alpha \beta}\, \lint_{\gamma(1;\theta)} %
\frac{ \re^{ \z^{1/(\alpha \beta)} } \z^{ \frac{\alpha + \beta + 1 - \mu}{\alpha \beta} - 1 %
+ r_\beta/\beta } } {x^{r_\beta} (\z^{1/\beta} - x)(\z^{1/\alpha} - y)}\, \rd \z\ %
\quad \tand \label{int2}\\
I_3 &= - \frac{1}{2\pi\ri} \frac{1}{\alpha \beta}\, \lint_{\gamma(1;\theta)}\, %
\frac{ \re^{ \z^{1/(\alpha \beta)} } \z^{ \frac{\alpha + \beta + 1 - \mu}{\alpha \beta} - 1 + %
r_\alpha/\alpha + r_\beta/\beta } } {x^{r_\beta} y^{r_\alpha}(\z^{1/\beta} - x)(\z^{1/\alpha} - y)}\, %
\rd \z. \label{int3}
\end{flalign}

Assuming that $|\arg x|\le \tau/\beta$\ and $|\arg y|\le \tau/\alpha$, each of the three integrals in the sum~\eqref{mlrep6}, $I_1$\ in~\eqref{int1}, $I_2$\ in~\eqref{int2} and $I_3$\ in~\eqref{int3}, can be evaluated for large values of $|x|$\ and $|y|$. Provided that $|\arg x|\le \tau/\beta$\ for $|x|$\ large enough, it can be checked that
\begin{equation*}
\min_{\z\in \gamma(1;\theta)} \l|\z^{1/\beta} - x\r| = |x| \sin(\theta/\beta - \tau/\beta) %
= |x| \sin\l(\tfrac{\theta - \tau}{\beta}\r)
\end{equation*}
and analogously, when $|\arg y|\le \tau/\beta$ for $|y|$ large enough,
\begin{equation*}
\min_{\z\in \gamma(1;\theta)} \l|\z^{1/\alpha} - y\r| = |y| \sin(\theta/\alpha - \tau/\alpha) %
= |y| \sin\l(\tfrac{\theta - \tau}{\alpha}\r).
\end{equation*}
Hence, when $|\arg x|\le \tau/\beta$\ and $|\arg y|\le \tau/\alpha$ for large $|x|$\ and $|y|$, the following estimate for the integral $I_1$ can be obtained:
\begin{equation} \label{estimateI1}
|I_1|\le \frac{ |x|^{-1} |y|^{-r_\alpha -1} }{ 2\pi \alpha \beta %
\sin\l(\frac{\theta - \tau}{\alpha}\r) \sin\l(\frac{\theta - \tau}{\beta}\r) }\, %
\lint_{\gamma(1;\theta)} \l|\re^{\z^{\frac{1}{\alpha \beta}} }\r| \l|\z^{ \frac{\alpha + \beta + 1 - \mu}{\alpha \beta} - 1 + \frac{r_\alpha}{\alpha} }\r|\, |\rd \z|.
\end{equation}
Of course, an analogous estimate holds also symmetrically for the integral $I_2$ by substituting $r_\beta$ for $r_\alpha$ and $r_\beta/\beta$ for $r_\alpha/\alpha$ (resp.) into inequality~\eqref{estimateI1}.

Besides, since the rays defined by $S_{\theta} = \bl\{\arg \z = \pm \theta,\ |\z|\ge 1\br\}$ belong to the contour $\gamma(1;\theta)$, the integral in inequality~\eqref{estimateI1} is convergent; whence the equality
\[
\dis \l|\re^{ \z^{1/(\alpha \beta)} }\r| = \exp\l(|\z|^{\frac{1}{\alpha \beta}} %
\cos \frac{\theta}{\alpha \beta} \r).\]
Now, according to inequalities~\eqref{cond2}, we have that $\cos \frac{\theta}{\alpha \beta} < 0$. Thus,
\[
I_1 = o\l(|xy|^{-1} |y|^{-r_\alpha}\r)\ \qquad \tand \qquad I_2 = o\l(|xy|^{-1} |x|^{-r_\beta}\r).\]

Furthermore, by referring to Eq.~\eqref{int3}, the next estimate is also obtained for the integral $I_3$.
\begin{equation} \label{estimateI3}
\kern-.7cm |I_3|\le \frac{ |x|^{-r_\beta - 1} |y|^{-r_\alpha -1} }{ 2\pi \alpha\beta %
\sin\l(\frac{\theta - \tau}{\alpha}\r) \sin\l(\tfrac{\theta - \tau}{\beta}\r) }\, %
\lint_{\gamma(1;\theta)} \l|\re^{ \z^{1/(\alpha \beta)} }\r| \l|\z^{ \frac{\alpha + \beta + 1 - \mu} %
{\alpha \beta} - 1 + r_\alpha/\alpha + r_\beta/\beta }\r|\, |\rd \z|,
\end{equation}
which yields the asymptotic formula $I_3 = o\l(|xy|^{-1} |x|^{-r_\beta} |y|^{-r_\alpha}\r)$.

Hence, this leads finally to the overall asymptotic formula
\[
I_1 + I_2 + I_3 = o\l(|xy|^{-1} |x|^{-r_\beta}\r) + o\l(|xy|^{-1} |y|^{-r_\alpha}\r)\]
and the proof of Eq.~\eqref{as1} (in case 1 of Theorem~\ref{main}) is established.

Similarly, the proofs of Eq.~\eqref{as2} (case 2), Eq.~\eqref{as3} (case 3) and Eq.~\eqref{as4} (case 4) run along the same lines as the above proof of Eq.~\eqref{as1} (case 1). This completes the proof of Theorem~\ref{main}.
\end{proof}

\textbf{Acknowledgments.} The author would like to thank an anonymous reviewer for his/her very careful reading and valuable suggestions. His/her remarks benefited to a great help in correcting several mistakes and improved the quality of the paper.

\bibliographystyle{amsplain}

\end{document}